\documentclass{amsart}%
\usepackage{eurosym}
\usepackage{amsmath}
\usepackage{amsfonts}
\usepackage{amssymb}
\usepackage{graphicx}
\usepackage{latexsym}
\usepackage{amscd}
\usepackage{tikz-cd}
\usepackage{enumitem}
\usepackage[all, cmtip]{xy}
\usepackage[utf8]{inputenc}
\usepackage{amsthm}
\usepackage{xcolor}
\usepackage[a4paper,  margin=1in]{geometry}
\usepackage{epic}
\usepackage{pstricks}
\usepackage[all]{xy}
\usepackage{graphicx}
\usepackage{epic,eepic}
\usepackage{epsfig}
\usepackage{float}
\usepackage{extarrows}
\usepackage{tikz}
\usepackage{calc}%
\providecommand{\U}[1]{\protect\rule{.1in}{.1in}}
\providecommand{\U}[1]{\protect\rule{.1in}{.1in}}
\providecommand{\U}[1]{\protect\rule{.1in}{.1in}}
\providecommand{\U}[1]{\protect\rule{.1in}{.1in}}
\newtheorem{theorem}{Theorem}[section]

\newtheorem{corollary}[theorem]{Corollary}

\newtheorem{definition}[theorem]{Definition}
\newtheorem{example}[theorem]{Example}

\newtheorem{lemma}[theorem]{Lemma}

\newtheorem{proposition}[theorem]{Proposition}
\newtheorem{remark}[theorem]{Remark}

\def\remove#1{}

\def\path{\mathop{\hbox{\rm Path}}}

\DeclareMathOperator{\Soc}{Soc}
\usetikzlibrary{decorations.markings}

\tikzstyle{vertex}=[circle, draw, fill, inner sep=0pt, minimum size=6pt]
\begin{document}
\title{Graded Naimark's Problem for Leavitt Path Algebras}
\author{Kulumani M. Rangaswamy}
\address{Department of Mathematics, University of Colorado, Colorado Springs, CO.
80918, USA}
\email{kmranga@gmail.com}
\author{Ashish K. Srivastava}
\address{Department of Mathematics and Statistics, Saint Louis University, St. Louis,
MO-63103, USA}
\email{ashish.srivastava@slu.edu}

\begin{abstract}
In this paper we study the graded version of Naimark's problem for Leavitt
path algebras considering them as $\mathbb{Z}$-graded algebras. Several
characterizations are obtained of a Leavitt path algebra $L$ of an arbitrary
graph $E$ over a field $\mathbb{K}$ over which any two graded-simple modules
are graded isomorphic. Such a Leavitt path algebra $L$ is shown to be graded
isomorphic to the algebra of graded infinite matrices having at most finitely
many non-zero entries from the ring $R$ where $R=\mathbb{K}$ or $R=\mathbb{K}%
[x,x^{-1}]$. Equivalently, $L$ is a graded-simple ring which is
graded-semisimple, that is, $L$ is a graded direct sum of graded-isomorphic
graded-simple left $L$-modules. Graphically, the graph $E$ is shown to be
row-finite, downward directed and the vertex set $E^{0}$ is the hereditary
saturated closure of a single vertex $v$ which is either a line point or lies
on a cycle without exits. We also characterize Leavitt path algebras
possessing at most countably many isomorphism classes of graded-simple left
modules. Examples are constructed illustrating these results.

\end{abstract}
\maketitle

\section{Introduction}

\noindent Let $\mathcal{K(H})$ denote the C*-algebra of compact operators on a
complex, not necessarily separable, Hilbert space $\mathcal{H}$. In 1948,
Naimark proved (see \cite{Naimark1}) that any two irreducible representations
of $\mathcal{K}(\mathcal{H})$ are unitarily equivalent and in \cite{Naimark2},
he asked whether this property characterizes the algebra $\mathcal{K}%
(\mathcal{H})$, that is, if $A$ is a C*-algebra with only one irreducible
representation up to unitary equivalence, should $A$ be isomorphic to some
$\mathcal{K}(\mathcal{H})$. \ Due to its potential impact in the
classification of C*-algebras, this problem attracted the attention of several
researchers whose efforts only lead to several partial solutions (see
\cite{ST} for some details). Specifically, Kaplansky (\cite{Kaplansky}) showed
that Naimark's question has a positive answer for Type I C*-algebras and in
1953, Rosenberg (\cite{Rosenberg}) proved that Naimark's problem has an
affirmative answer for separable C*-algebras. In spite of the attempts by many
researchers such as Diximier \cite{D-0}, Fell \cite{F}, Glimm \cite{GL} and
others, Naimark's question remained unsolved during the following 50
years.\ At last in 2004, Akemann and Weaver (\cite{AW}) used Jensen's diamond
axiom, which is a combinatorial principle independent of ZFC, to construct an
$\aleph_{1}$-generated C*-algebra as a counter example to Naimark's problem.
In particular, their work showed that \ a solution to Naimark's problem cannot
be obtained using ZFC alone. In view of the set-theoretic obstructions to
solve the Naimark's problem in general, it makes sense to consider his problem
for restricted classes of C*-algebras. Accordingly, Suri and Tomforde
\cite{ST} considered the case of graph C*-algebras and showed that the
Naimark's question has a affirmative answer for a graph C*-algebra $A$
provided $A$ is an AF-algebra. They left the question open whether Naimark's
question can be answered for other types of graph C*-algebras. Recently, the
first named author and Tomforde \cite{RT} have shown that Naimark's problem
indeed has an affirmative answer for arbitrary graph $C^{\ast}$-algebras and
they have also shown that the algebraic analogue of Naimark's problem has a
positive solution for Leavitt path algebras of arbitrary graphs.

Leavitt path algebras of an arbitrary graph over a field admit a natural
$\mathbb{Z}$-grading. In this paper, we consider the graded version of the
Naimark's problem for Leavitt path algebras. We first note that, in the
context of an associative algebra $A$, irreducible representations of $A$ can
be realized as simple left $A$-modules and that unitary equivalence of
irreducible representations is equivalent to isomorphisms of simple modules.
We prove the following main theorem:

\begin{theorem}
Suppose $L=L_{\mathbb{K}}(E)$ is a Leavitt path algebra of an arbitrary graph
$E$ over a field $\mathbb{K}$. Then the following properties are equivalent
for $L$:

(a) Any two graded-simple left $L$-modules are graded-isomorphic.

(b) The graph $E$ is row-finite, downward directed and $E^{0}$ is the
hereditary saturated closure of a single vertex $v$ which is either a line
point, or lies on a cycle without exits in $E$.

(c) $L\cong_{gr}M_{\Lambda}(\mathbb{K})(\bar{p})$ or $L\cong_{gr}M_{\Upsilon
}(\mathbb{K}[x^{m},x^{-m}])(\bar{q})$ with appropriate matrix gradings, where
$\Lambda$, $\Upsilon$ are suitable index sets.

\end{theorem}

We also completely describe the Leavitt path algebras $L$ having at most
countably many non-isomorphic graded-simple modules. In this case, $L$ is
shown to have a graded-socular chain of countable length, that is, $L$ is the
union of a smooth well-ordered ascending chain of graded ideals
\[
0\subseteq I_{1}\subseteq\cdot\cdot\cdot\subseteq I_{\alpha}\subseteq
I_{\alpha+1}\subseteq\cdot\cdot\cdot\qquad\qquad(\alpha<\tau)
\]
where $\tau$ is a countable ordinal and for each $0\leq\alpha<\tau$,
$I_{\alpha+1}/I_{\alpha}\cong_{gr}(\oplus_{i\in X}M_{\Lambda_{i}}%
(\mathbb{K})(\bar{\delta}_{i}))\oplus(\oplus_{j\in Y}M_{\Upsilon_{j}%
}(\mathbb{K}[x^{m_{j}},x^{-m_{j}}])(\bar{\sigma}_{j}))$ with appropriate
matrix gradings, where $\Lambda_{i},\Upsilon_{j}$ are suitable index sets and
$X,Y $ are at most countable. Several graphical examples are given
illustrating these results.

\section{Preliminaries}

\noindent We begin this section by recalling some useful notions of graph
theory. We refer to \cite{AAS} for the general results, notation and
terminology in Leavitt path algebras. A (\textit{directed}) \textit{graph} is
a quadruplet $E=(E^{0},E^{1},s,r)$ which consists of two disjoint sets $E^{0}$
and $E^{1}$, called the set of \emph{vertices} and the set of \emph{edges}
respectively, together with two maps $s,r:E^{1}\longrightarrow E^{0}$. The
vertices $s(e)$ and $r(e)$ are referred to as the \emph{source} and the
\emph{range} of the edge~$e$, respectively. A vertex~$v$ for which $s^{-1}(v)$
is empty is called a \emph{sink}; a vertex~$v$ is \emph{regular} if
$0<|s^{-1}(v)|<\infty$; a vertex~$v$ is an \textit{infinite emitter} if
$|s^{-1}(v)|=\infty$; and a vertex is \textit{singular} if it is either a sink
or an infinite emitter.

In this paper, all the graphs\ $E$ that we consider are arbitrary in the sense
that there are no restrictions on the the number of vertices and \ edges in
$E$ and no restriction on the number of edges emitted by a single vertex. The
graph $E$ is said to be \textit{row-finite}, if every vertex in $E$ is a
regular vertex.

A \emph{finite path of length} $n$ in a graph $E$ is a sequence $p = e_{1}
\cdots e_{n}$ of edges $e_{1}, \dots, e_{n}$ such that $r(e_{i}) = s(e_{i+1})
$ for $i = 1, \dots, n-1$. In this case, we say that the path~$p$ starts at
the vertex $s(p) := s(e_{1})$ and ends at the vertex $r(p) := r(e_{n})$, we
write $|p| = n$ for the length of $p$. We consider the elements of $E^{0}$ to
be paths of length $0$. We denote by $E^{*}$ the set of all finite paths in
$E$. An edge $f$ is an \textit{exit} for a path $p= e_{1} \cdots e_{n}$ if
$s(f) = s(e_{i})$ but $f \neq e_{i}$ for some $1\le i\le n$. A finite path $p$
of positive length is called a \textit{closed path based at} $v$ if $v = s(p)
= r(p)$. A \textit{cycle} is a closed path $p = e_{1} \cdots e_{n}$, and for
which the vertices $s(e_{1}), s(e_{2}), \hdots, s(e_{n})$ are distinct. A
closed path $c$ in $E$ is called \textit{simple} if $c \neq d^{n}$ for any
closed path $d$ and integer $n\ge2$.

\begin{definition}
\label{LPA} For an arbitrary graph $E = (E^{0},E^{1},s,r)$ and any field
$\mathbb{K}$, the \emph{Leavitt path algebra} $L_{\mathbb{K}}(E)$ \textit{of
the graph}~$E$ \emph{with coefficients in}~$\mathbb{K}$ is the $\mathbb{K}%
$-algebra generated by the union of the set $E^{0}$ and two disjoint copies of
$E^{1}$, say $E^{1}$ and $\{e^{*}\mid e\in E^{1}\}$, satisfying the following
relations for all $v, w\in E^{0}$ and $e, f\in E^{1}$:

\begin{itemize}
\item[(1)] $vw=\delta_{v,w}w$.

\item[(2)] $s(e) e = e = e r(e)$ and $e^{*}s(e) = e^{*} = r(e) e^{*}$;

\item[(3)] $e^{*} f = \delta_{e, f} r(e)$;

\item[(4)] $v= \sum_{e\in s^{-1}(v)}ee^{*}$ for any regular vertex $v$;
\end{itemize}

where $\delta$ is the Kronecker delta.
\end{definition}

If $E^{0}$ is finite, then $L_{\mathbb{K}}(E)$ is a unital ring having
identity $1 = \sum_{v\in E^{0}}v$. It is easy to see that the mapping given by
$v\longmapsto v$ for all $v\in E^{0}$, and $e\longmapsto e^{*}$,
$e^{*}\longmapsto e$ for all $e\in E^{1}$, produces an involution on the
algebra $L_{\mathbb{K}}(E)$, and for any path $p = e_{1}e_{2}\cdots e_{n}$,
the element $e^{*}_{n}\cdots e^{*}_{2}e^{*}_{1}$ of $L_{K}(E)$ is denoted by
$p^{*}$. It can be shown that $L_{\mathbb{K}}(E)$ is spanned as a $\mathbb{K}%
$-vector space by $\{pq^{*} \mid p, q\in E^{*}, r(p) = r(q)\}$. Indeed,
$L_{\mathbb{K}}(E)$ is a $\mathbb{Z}$-graded $\mathbb{K}$-algebra:
$L_{\mathbb{K}}(E) = \oplus_{n\in\mathbb{Z}}L_{\mathbb{K}}(E)_{n}$, where for
each $n\in\mathbb{Z}$, the degree $n$ component $L_{\mathbb{K}}(E)_{n}$ is the
set $\text{span}_{\mathbb{K}} \{pq^{*}\mid p, q\in E^{*}, r(p) = r(q), |p|-
|q| = n\}$. Also, $L_{\mathbb{K}}(E)$ has the following property: if
$\mathcal{A}$ is a $\mathbb{K}$-algebra generated by a family of elements
$\{a_{v}, b_{e}, c_{e^{*}}\mid v\in E^{0}, e\in E^{1}\}$ satisfying the
relations analogous to (1) - (4) in Definition~\ref{LPA}, then there exists a
unique $\mathbb{K}$-algebra homomorphism $\varphi: L_{\mathbb{K}%
}(E)\longrightarrow\mathcal{A}$ given by $\varphi(v) = a_{v}$, $\varphi(e) =
b_{e}$ and $\varphi(e^{*}) = c_{e^{*}}$. We will refer to this property as the
Universal Property of $L_{\mathbb{K}}(E)$.

\begin{definition}
Let $E$ be an arbitrary graph and $H\subseteq E^{0}$.

(i) $H$ is called a\textit{\ }\textbf{hereditary set }if whenever $u\in H$ and
there is a path from $u$ to a vertex $v$ (in symbols, $u\geq v$) then $v\in H$.

(ii) $H$ is called a \textbf{saturated set }if, for any regular vertex $v$,
$r(s^{-1}(v))\subseteq H$ implies that $v\in H$.

(iii) If $X$ is a non-empty subset of $E^{0}$, the \textbf{hereditary
saturated closure} of $X$ is the intersection of all the hereditary saturated
subsets of $E^{0}$ (and thus the smallest hereditary saturated subset)
containing $X$.
\end{definition}

\noindent If $v\in E^{0}$, the \textbf{tree of }$v$ is the set $T(v)=\{w\in
E^{0}:v\geq w\}$. We say a vertex $v$ has\textbf{\ bifurcation }or $v$ is a
\textbf{bifurcation vertex} if $v$ emits two or more edges. For any vertex
$v\in E$, the\textbf{\ }\textit{tree} $T(v)=\{w\in E^{0}:v\geq w\}$. A vertex
$v$ is called a \textbf{line point }if $T(v)$ contains no bifurcating vertices
and no cycles. Thus $T(v)$, when we add all the relevant edges, gives rise to
a straight line segment shown as below
\[
\bullet_{v=v_{1}} \longrightarrow\bullet_{v_{2}}\longrightarrow\bullet_{v_{3}%
}\longrightarrow\bullet\longrightarrow\bullet\cdots
\]
In particular, a sink is a line point. A vertex $v$ is called a\textit{\ }%
\textbf{Laurent vertex} if $T(v)$ is the set of all vertices on a single path
$\gamma=\mu c$ with $s(\gamma)=v$, where $\mu$ is a path without bifurcations
and $c$ is a cycle without exits.

The next lemma describes the explicit construction of the hereditary saturated
closure of a set of vertices.

\begin{lemma}
\label{SaturatedClosure}$($Lemma 2.0.7, \cite{AAS}$)$Let $X$ is a non-empty
subset of vertices. Then its hereditary saturated closure $\overline{X}%
=\cup_{n<\omega}X_{n}$, where $X_{n}$ is defined inductively as follows:
\[
X_{0}=\cup_{v\in X}T(v).
\]
If $X_{k}$ has been defined for some $k\geq0$, the define $X_{k+1}=X_{k}%
\cup\{v\in E^{0}:v$ is regular and $r(s^{-1}(v))\subset X_{k}\}$.
\end{lemma}

The following remark will be used in the proof of our main theorem..

\begin{remark}
\label{vertex in sat closure}From Lemma \ref{SaturatedClosure}, it is then
clear that if $\overline{X}$ is the saturated closure of a hereditary subset
$X$ of $E^{0}$, then every vertex $u\in\overline{X}\backslash X$ is a regular vertex.
\end{remark}

\begin{definition}
A graded left module $M$ over a graded ring $R$ is called a
\textbf{graded-simple module} if $M$ has no non-zero proper graded submodules.
\end{definition}

\noindent For example, in the usual $\mathbb{Z}$-grading (with $deg(x)=1$,
$deg(x^{-1})=-1$), $\mathbb{K}[x,x^{-1}]$ is graded-simple as a $\mathbb{K}%
[x,x^{-1}]$-module but is not simple.

First, we recall the constructions of some of the graded-simple left
$L_{\mathbb{K}}(E)$-modules obtained by using the vertices and infinite paths
in the graph $E$ and list two useful theorems from \cite{AR}, \cite{C},
\cite{HR}.

Two infinite paths $p=e_{1}e_{2}\cdot\cdot\cdot e_{n}\cdot\cdot\cdot$ and
$q=f_{1}f_{2}\cdot\cdot\cdot f_{n}\cdot\cdot\cdot$ in a graph $E$ are said to
be \textit{\ }\textbf{tail-equivalent}, if there exist positive integers $m,n
$ such that $e_{m+i}=f_{n+i}$ for all $i\geq0$. Tail-equivalence among
infinite paths is an equivalence relation and the equivalence class of all
paths tail-equivalent to an infinite path $p$ is denoted by $[p]$.

An infinite path $p$ is said to a \textbf{rational path}\textit{\ }if it is
tail equivalent to an infinite path $q=ccc\cdot\cdot\cdot$, where $c$ is a
closed path. An infinite path which is not rational is called an
\textbf{irrational path}.

\begin{theorem}
\label{Chen} \cite{C} Let $E$ be an arbitrary graph and let $L=L_{\mathbb{K}%
}(E) $.

1. Given an infinite path $p$ in $E$, let $V_{[p]}$ denote the vector space
over the field $\mathbb{K}$ having the set $\{q\in\lbrack p]\}$ as a basis.

(a)Then $V_{[p]}$ becomes a left $L$-module induced by the following action on
each path $q\in\lbrack p]$:

\qquad(i) For all vertices $u\in E^{0}$, $u\cdot q=q$ or $0$ according as
$u=s(q)$ or not.

\qquad(ii) For all edges $e$ in $E$, $e\cdot q=eq$ if $r(e)=s(q)$, otherwise
$e\cdot q=$ $0$; and $e^{\ast}\cdot q=q_{1}$ if $q=eq_{1}$, otherwise,
$e^{\ast}\cdot q=0$.

(b) $V_{[p]}$ is a simple $L$-module.

(c) Given two infinite paths $p,q$, $V_{[p]}\cong V_{[q]}$ \ if and only if
$p,q$ are tail-equivalent.

2. Suppose $w$ is a sink in $E$. Let $N_{w}$ denote the $\mathbb{K}$-vector
space having as a basis the set $\{p:p$ path in $E$ with $r(p)=w\}$. Then
$N_{w}$ becomes a simple left $L$-module under action of $L$ as defined above
in 1(a).
\end{theorem}

The next result points out which of the above-mentioned simple $L$-modules are
graded-simple modules.

\begin{theorem}
\label{HR}(\cite{HR}, Examples A,B,C; Proposition 3.6) Let $E$ be an arbitrary
graph and $L=L_{\mathbb{K}}(E)$.

(i) The simple left $L$-modules $N_{w}$ and $S_{v\infty}$ are graded-simple modules.

(ii) If $p=ccc\cdot\cdot\cdot$ is an infinite rational path in $E$, then the
corresponding simple left $L$-module $V_{[p]}$ is not graded-simple.

(iii) If $p$ is an irrational infinite path, then the simple module $V_{[p]}$
is a graded-simple $L$-module.

(iv) If a vertex $v$ in $E$ is a line point or a Laurent vertex, then the left
ideal $Lv$ is a graded-simple left-ideal of $L$.
\end{theorem}

\section{Grading of Matrices}

\noindent In this section, we recall the grading of the rings\ of matrices of
finite order and also the rings of infinite matrices having at most finitely
many non-zero entries (see \cite{H}, \cite{HR}, \cite{NO}). This will be used
in the main theorem of the next section.

Let $\Gamma$ be an additive abelian group and $A=\oplus_{\lambda\in\Gamma
}A_{\lambda}$ be a $\Gamma$-graded ring. Fix an $n$-tuple $(\delta_{1}%
,\cdot\cdot\cdot,\delta_{n})$, where $\delta_{i}\in\Gamma$. Then the matrix
ring $M_{n}(A)$ can be made a $\Gamma$-graded ring $M_{n}(A)=\oplus
_{\lambda\in\Gamma}M_{n}(A)_{\lambda}(\delta_{1},\cdot\cdot\cdot,\delta_{n})$
by defining, for each $\lambda\in\Gamma$, its $\lambda$-homogeneous component
as the following additive subgroup $M_{n}(A)_{\lambda}(\delta_{1},\cdot
\cdot\cdot,\delta_{n})$ of $M_{n}(A)$ consisting of $n\times n$ matrices

$M_{n}(A)_{\lambda}(\delta_{1},\cdot\cdot\cdot\delta_{n})=$ $\left(
\begin{array}
[c]{cccccc}%
A_{\lambda+\delta_{1}-\delta_{1}} & A_{\lambda+\delta_{2}-\delta_{1}} & \cdot
& \cdot & \cdot & A_{\lambda+\delta_{n}-\delta_{1}}\\
A_{\lambda+\delta_{2}-\delta_{2}} & A_{\lambda+\delta_{2}-\delta2} &  &  &  &
A_{\lambda+\delta_{n}-\delta_{2}}\\
\cdot &  &  &  &  & \cdot\\
\cdot &  &  &  &  & \cdot\\
A_{\lambda+\delta_{1}-\delta_{n}} & A_{\lambda+\delta_{2}-\delta_{n}} &  &  &
& A_{\lambda+\delta_{n}-\delta_{n}}%
\end{array}
\right)  \qquad\qquad\qquad(1)$

This shows that for each homogeneous element $x\in A$, the degree of the
matrix $e_{ij}(x)$ with $x$ in the $ij$-position and with every other entry
$0$ is given by
\[
\deg(e_{ij}(x))=\deg(x)+\delta_{i}-\delta_{j}\text{.\qquad\qquad\qquad(2)}%
\]

\begin{lemma}
\label{RepresentFinite matrix}(Lemma\ I.5.4, \cite{NO}) that $M_{n}%
(A)(\delta_{1},\cdot\cdot\cdot,\delta_{n})$ is graded isomorphic to the ring
$End_{A}(M)$ of a graded free $A$-module with a homogeneous basis
$\{b_{1},\cdot\cdot\cdot,b_{n}\}$ with $\deg(b_{i})=\delta_{i}$, for
$i=1,\cdot\cdot\cdot,n$.
\end{lemma}

Next, let $A$ be a $\Gamma$-graded ring and let $I$ be an arbitrary infinite
index set. Consider the ring $M_{I}(A)$ consisting of square matrices whose
rows and columns are indexed by $I$ such that at most finitely many entries
are non-zero. Now fix a \textquotedblleft vector" $\ \bar{\delta}=(\cdot
\cdot\cdot,\delta_{i},\cdot\cdot\cdot)_{i\in I}$ where $\delta_{i}\in\Gamma$.
Following the same way used for grading matrices of finite order $M_{I}(A)$
can be made a $\Gamma$-graded ring. As before, for each $(i,j)\in I\times I$
and homogeneous element $x\in A$, $e_{ij}(x)$ denotes the matrix having $x$ at
the $ij$ position and every other entry being $0$. Under the stated grading,
$\deg(e_{ij}(x))=\deg(x)+\delta_{i}-\delta_{j}$. This makes $M_{I}(A)$ a
$\Gamma$-graded ring which we denote by $M_{I}(A)(\bar{\delta})$. Clearly, if
$I$ is finite with $|I|=n$, then the usual graded matrix ring $M_{n}(A)$
coincides (after a suitable permutation) with $M_{n}(A)(\delta_{1},\cdot
\cdot\cdot,\delta_{n})$.

Now, let us translate the above in the context of Leavitt path algebras.
Suppose $E$ is a finite acyclic graph consisting of exactly one sink $v$. Let
$\{p_{i}:1\leq i\leq n\}$ be the set of all paths ending at $v$. Then it was
shown in (Lemma 3.4, \cite{AAS})%
\[
L_{\mathbb{K}}(E)\cong M_{n}(\mathbb{K})\qquad\qquad\qquad(3)
\]
under the map
\[
p_{i}p_{j}^{\ast}\longmapsto e_{ij}\text{ \qquad\qquad\qquad(4)}%
\]

Now taking into account the grading, it was further shown in (Theorem
4.14,\cite{H-1}) that the same map (4) induces a graded isomorphism
\[
L_{\mathbb{K}}(E)\longrightarrow M_{n}(\mathbb{K})(|p_{1}|,\cdot\cdot
\cdot,|p_{n}|)\qquad\qquad\qquad(5).
\]
In the case of a comet graph $E$ (that is, a finite graph $E$ in which every
path eventually ends at a vertex on a cycle $c$ without exits), it was shown
in \cite{AAS-1} that the map%

\[
p_{i}c^{k}p_{j}^{\ast}\longmapsto e_{ij}(x^{k})
\]
induces an isomorphism
\[
L_{\mathbb{K}}(E)\longrightarrow M_{n}(\mathbb{K}[x,x^{-1}])\qquad\qquad
\qquad(6).
\]
Again taking into account the grading, it was shown in (Theorem 4.20,
\cite{H-1}) that the map%

\[
p_{i}c^{k}p_{j}^{\ast}\longmapsto e_{ij}(x^{k|c|})
\]
induces a graded isomorphism%
\[
L_{\mathbb{K}}(E)\longrightarrow M_{n}(\mathbb{K}[x^{|c|},x^{-|c|}%
])(|p_{1}|,\cdot\cdot\cdot\cdot,|p_{n}|)\qquad\qquad\qquad(7)
\]
Later in the paper \cite{AAPS}, the isomorphisms (3) and (6) were extended to
infinite acyclic and infinite comet graphs respectively (see Proposition 3.6,
\cite{AAPS}). The same isomorphisms with the grading adjustments will induce
graded isomorphisms for Leavitt path algebras of such graphs.

We now describe this extension below.

\begin{theorem}
\label{Hazrat}$($Theorems 4.14, 4.20, \cite{H-1}$)$ (i) Suppose $E$ is a
finite acyclic graph containing exactly one sink $w$ (so every path in $E$
eventually ends at $w$). Let $\{p_{1},\cdot\cdot\cdot,p_{n}\}$ be the set of
all paths in $E$ that end at $w$. Then there is a $\mathbb{Z}$-graded
isomorphism $L_{\mathbb{K}}(E)\longrightarrow M_{n}(\mathbb{K})(|p_{1}%
|,\cdot\cdot\cdot,|p_{n}|)$ induced by mapping $p_{i}p_{j}^{\ast}$ to $e_{ij}$.

(ii) Suppose $E$ is a finite graph with exactly one cycle $c$ and $c$ is of
length $m$, is based on a vertex $u$ and has no exits. Suppose further that
every path in $E$ ends at a vertex on $c$. If $p_{1},\cdot\cdot\cdot,p_{n}$ is
a listing of all the paths that end at $u$, but do not go through the entire
cycle $c$, then there is a $\mathbb{Z}$-graded isomorphism
\[
L_{\mathbb{K}}(E)\longrightarrow M_{n}(\mathbb{K}[x^{m},x^{-m}])(|p_{1}%
|,\cdot\cdot\cdot,|p_{n}|)
\]
induced by mapping $p_{i}c^{k}p_{j}^{\ast}$ to $e_{ij}(x^{km})$.
\end{theorem}

Theorem \ref{Hazrat} was extended to the case when $E$ is a row-finite graph
in \cite{HR}. We list it as a theorem below.

\begin{theorem}
\label{Generalized Hazrat} $($Theorem 6.7, \cite{HR}$)$ Let $E$ be a
row-finite graph which is downward directed and $E^{0},\emptyset$ are the only
hereditary saturated subsets of $E^{0}$.

\begin{enumerate}
\item Suppose $E$ contains a line point $v$ with $T(v)=p$. Let $\{p_{i}%
:i\in\Lambda\}$ be the set of all paths $p_{i}$ in $E$ which meets $T(v)$ for
the first time at $r(p_{i})$. Let $\bar{p}$ be the ``vector" $\bar{p}%
=(\cdot\cdot\cdot,|p_{i}|,\cdot\cdot\cdot)_{i\in\Lambda}$. Then the graded
isomorphism in Theorem \ref{Hazrat}(i) extends to a $\mathbb{Z}$-graded
isomorphism $L_{K}(K)\cong_{gr}M_{\Lambda}(K)(\bar{p})$.

\item Suppose $E$ contains a cycle $c$ with no exits, based at a vertex $u$
and has length $m$. Let $\{q_{j}:j\in\Upsilon\}$ be the set of all paths that
end at $u$, but do not through the entire cycle $c$. Let $\bar{q}$ be the
\textquotedblleft vector" $\bar{q}=(\cdot\cdot\cdot,|q_{j}|,\cdot\cdot
\cdot)_{j\in\Upsilon}$. Then the graded isomorphism in Theorem \ref{Hazrat}%
(ii) extends to a $\mathbb{Z}$-graded isomorphism $L_{\mathbb{K}}(E)\cong%
_{gr}M_{\Upsilon}(\mathbb{K}[x^{m},x^{-m}])(\bar{q})$.
\end{enumerate}
\end{theorem}

In the proof of the main theorems in the next two sections, we will be using
the following result from \cite{HR}

\begin{theorem}
\label{graded socle} (\cite{HR}, Theorem 2.10) Let $E$ be an arbitrary graph,
$\mathbb{K}$ be any field and $L=L_{\mathbb{K}}(E)$. Then the two-sided ideal
generated by the set consisting of all the line points and all the vertices on
cycles without exits in $E$ is the graded socle $\Soc_{gr}(L)$ given by
\[
Soc_{gr}(L)\cong_{gr}(\oplus_{i\in X}M_{\Lambda_{i}}(\mathbb{K})(\bar{\delta
}_{i}))\oplus(\oplus_{j\in Y}M_{\Upsilon_{j}}(\mathbb{K}[x^{m_{j}},x^{-m_{j}%
}])(\bar{\sigma}_{j}))
\]
with appropriate matrix gradings.
\end{theorem}

\section{Leavitt Path Algebras with exactly one isomorphism class of
graded-simple modules}

\noindent Let $E$ be an arbitrary graph and let $L:=L_{\mathbb{K}}(E)$ be the
Leavitt path algebra of the graph $E$ over a field $\mathbb{K}$. In this
section, we obtain a complete characterization of those Leavitt path algebras
$L$ over which any two graded-simple modules are graded-isomorphic. In this
case, the graph $E$ is shown to be row-finite, downward directed and $E^{0}$
is the hereditary saturated closure of a line point or of a vertex on a cycle
without exits.

Note that graded-simple left modules over $L=L_{\mathbb{K}}(E)$ always exist.
Indeed, if $u\in E^{0}$, use Zorn's Lemma to find a maximal graded left
$L$-submodule $N$ of $Lu$. Then $M=N\oplus(\oplus_{v\in E^{0},v\neq u}Lv)$ is
a maximal graded left ideal of $L$ and so $L/M$ is a graded-simple left $L$-module.

We are now ready to consider the graded version of Naimark's problem for the
$\mathbb{Z}$-graded Leavitt path algebra $L_{\mathbb{K}}(E)$.

\begin{theorem}
\label{Graphical conditions}\textit{Let }$E$\textit{\ be an arbitrary graph
and let }$L:=L_{\mathbb{K}}(E)$\textit{. If any two graded-simple left }%
$L$\textit{-modules are graded isomorphic, then the graph }$E$\textit{\ is
row-finite, downward directed and }$E^{0},\emptyset$\textit{\ are the only
hereditary saturated subsets of }$E^{0}$\textit{. }
\end{theorem}

\textbf{Proof}: Since all the graded-simple left $L$-modules are graded
isomorphic, there can be only one annihilating ideal of the graded-simple left
modules in $L$. Now the graded Jacobson radical $J_{gr}(L)$ is the
intersection of all the annihilators of the graded-simple left $L$-modules
(Lemma 1.7.4 \cite{NO}). Also, by Proposition 1.1.33, \cite{H}, the graded
Jacobson radical $J_{gr}(L)=\{0\}$. So $\{0\}$ is the only annihilator of
graded-simple left $L$-modules. We claim that $L$ is a graded-simple ring,
that is, $L$ has no non-zero proper graded ideals. Indeed, if $I$ is a
non-zero proper graded ideal of $L$, then a graded-simple left module\ over
the Leavitt path algebra $L/I$ will also be a graded-simple left $L$-module
whose annihilating ideal $Q$ contains $I$ and hence is non-zero. This
contradiction shows that $L$ is a graded-simple ring. In particular, $L$ is a
graded-prime ring and so the graph $E$ is downward directed (Theorem 2.4,
\cite{ABR}). By graded-simplicity of $L$, $E^{0}$ and $\emptyset$ are the only
hereditary saturated subsets of $E^{0}$.

We claim that $E$ is row-finite. Suppose, by way of contradiction, $E$
contains an infinite emitter $w$. First note $v\geq w$ for every vertex $v\in
E^{0}$. To see this, observe that $E^{0}$ is the hereditary saturated closure
of $T(v)$ and since $w$ an infinite emitter, Remark
\ref{vertex in sat closure} implies that $w\in T(v)$ and hence $v\geq w$.
Consequently, for each of the infinitely many edges $e_{i}$ with $s(e_{i})=w$,
we have $r(e_{i})\geq w$ and this will give rise to infinitely many distinct
closed paths $c_{i}$ based at $w$. Let $c_{1},c_{2},\cdot\cdot\cdot
,c_{n},\cdot\cdot\cdot$ be the listing of a countable set of distinct closed
paths based at $w$ corresponding to a countable subset of edges $e_{i}$ with
$s(e_{i})=w$ We then obtain two infinite irrational paths $p=c_{1}c_{3}%
c_{5}\cdot\cdot\cdot$ and $q=c_{2}c_{4}c_{6}\cdot\cdot\cdot$ which are not
tail-equivalent. Then, by Theorem \ref{Chen}(c) and Theorem \ref{HR}(ii), we
obtain two non-isomorphic (hence non-graded-isomorphic) graded-simple left
$L$-modules. This contradiction shows that the graph $E$ must be row finite.

The next Lemma from \cite{RT} is useful in proving Theorem
\ref{Only one graded-simple} and also in the next section.

\begin{lemma}
\label{cofinal} Suppose the graph $E$ contains no sinks. Suppose $q=e_{1}%
e_{2}e_{3}....$ is an infinite \ path in $E$ with $s(e_{i})=v_{i}$ for all
$i\geq1$such that no edge on $q$ and no vertex on $q$ is part of a closed
path, and $q$ contains infinitely many bifurcating vertices. Then given any
vertex $v\in E^{0}$, there is a finite path $\gamma=f_{1}\cdot\cdot\cdot
f_{n}$ for some $n\geq1$ \ such that $v=s(f_{1})$ and the last edge $f_{n}\neq
e_{i}$ for all $i\geq1$ and $f_{n}$ is not part of a closed path.
\end{lemma}

\begin{proof}
Since $E$ has no line points, $E$ has no sinks. Since $v$ is not a sink, there
exists an edge $f_{1}$ with $s(f_{1})=v$. If $f_{1}$ is not an edge on $q$, we
may let $\gamma:=f_{1}$. If instead $f_{1}$ is an edge on $q$, then, for some
$n>1$, we may extend the edge $f_{1}$ to a finite path $f_{1}\cdot\cdot\cdot
f_{n-1}$ on $q$ with $r(f_{n-1})$ is a bifurcating vertex on $q$ , say $v_{j}$
for some $j\geq1$. Denote the bifurcating edge at $v_{j}$ by $f_{n}$. Since
the path $q$ does not go through any closed paths, $f_{n}\neq e_{i}$ for all
$i\geq1$. Then the path $\gamma:=f_{1}\cdot\cdot\cdot f_{n}$ satisfies the
desired condition.
\end{proof}

The following Lemma \ref{Disjoint cycles} is\ also used in the next section.
Its first statement (i) is from \cite{AR} and, for the sake of completeness,
we give its proof.

\begin{lemma}
\label{Disjoint cycles} If $L=L_{\mathbb{K}}(E)$ has at most countably many
non-isomorphic graded simple left modules, then we have the following.

(i) Distinct cycles in $E$ are disjoint, that is, they have no common vertex.

(ii) If $p$ is an infinite irrational path, then at most finitely many
vertices on $p$ are bases of finitely many distinct closed paths and thus
$p=\gamma q$ where $q$ is an infinite irrational path containing infinitely
many bifurcating vertices, and no vertex on the path $q$ lies on a closed path
and $\gamma$ is a finite path containing \ at most finitely many closed paths.
\end{lemma}

\begin{proof}
(i) Suppose, on the contrary, there are two cycles $c,d$ in $E$ having a
common vertex. We wish to construct uncountably many infinite irrational paths
in $E$. For each infinite subset $S$ of $\mathbb{N}$, define an infinite path
$p_{S}=p_{1}p_{2}p_{3}\cdot\cdot\cdot$ where, for each $i\in\mathbb{N}$,
$p_{i}=c$ or $d$ according as $i\in S$ or not. Then the set $\{p_{S}:S$ an
infinite subset of $\mathbb{N}\}$ consists of an uncountable number of
infinite irrational paths in $E$ no two of which are tail-equivalent. Then, by
Theorem \ref{Chen}(c) and Theorem \ref{HR}(ii), there will be uncountably many
pair-wise non-isomorphic graded-simple $L$-modules, a contradiction. This
proves (i).

(ii) Suppose, by way of contradiction, there are infinitely many distinct
closed paths $c_{1}c_{2}\cdot\cdot\cdot c_{n}\cdot\cdot\cdot$ each of which is
based at a vertex on the path $p$. Then we can construct an infinite
irrational path $q$ passing through each of these paths $c_{i}$. For each
$i\in\mathbb{N}$, suppose the path $q$ goes through $c_{i}$, $k_{i} $ times,
where the $k_{i}$ are positive integers. For each infinite subset $S $ of
$\mathbb{N}$, define an infinite irrational path $q_{S}$ which goes through
the closed path $c_{i}$, $k_{i}+1$ times if $i\in S$ and $k_{i}$ times if
$i\notin S$. Then we obtain uncountably many infinite irrational paths $q_{S}$
which are mutually not trail-equivalent and this give rise to uncountable
number of pair-wise non-isomorphic graded-simple \ $L$-modules, a
contradiction. Thus at most a finite number of vertices on $p$ are bases of
closed paths. Clearly, we can then write $p=\gamma q$ where $q$ is an infinite
irrational path containing infinitely many bifurcating vertices, and no vertex
on the path $q$ lies on a closed path and $\gamma$ is a finite path containing
\ at most finitely many closed paths. This proves (ii).
\end{proof}





Now we are ready to prove the main theorem of this section.

\begin{theorem}
\label{Only one graded-simple} Suppose $L=L_{\mathbb{K}}(E)$ is a Leavitt path
algebra of an arbitrary graph $E$ over a field $\mathbb{K}$. Then the
following properties are equivalent for $L$:

(a) Any two graded-simple left $L$-modules are graded-isomorphic.

(b) The graph $E$ is row-finite, downward directed and $E^{0}$ is the
hereditary saturated closure of a single vertex $v$ which is either a line
point, or lies on a cycle without exits in $E$.

(c) $L\cong_{gr}M_{\Lambda}(\mathbb{K})(\delta)$ or $L\cong_{gr}M_{\Upsilon
}(\mathbb{K}[x^{m},x^{-m}])(\bar{\delta})$ with appropriate matrix gradings,
where $\Lambda$, $\Upsilon$ are suitable index sets.

\end{theorem}

\begin{proof}
Assume (a). By Theorem \ref{Graphical conditions}, the graph $E$ is
row-finite, downward directed and $E^{0}$ and $\emptyset$ are the only
hereditary saturated subsets of vertices in $E$.

We wish to show that $E$ contains either a line point or a cycle without
exits. If $E$ contains a cycle $c$ without exits based at a vertex $v$, or if
$E$ contains a sink $v$, then the hereditary saturated closure of
$\{v\}=E^{0}$ and we are done. So assume that every cycle in $E$ has an exit
and that $E$ contains no sinks.

Then every vertex in $E$ is on an infinite path. Since every cycle has an
exit, infinite irrational paths do exist in $E$. By Theorem \ref{Chen},
infinite irrational paths in $E$ give rise to simple left $L$-modules, each of
which is graded-simple by Theorem \ref{HR}. Also, by Theorem \ref{Chen}, if
there are two infinite irrational paths which are not tail-equivalent, they
would give rise to non-isomorphic (hence not graded-isomorphic) graded-simple
modules. By hypothesis, any two graded-simple left $L$-modules are
graded-isomorphic and so all the infinite irrational paths in $E$ form a
single tail-equivalent class.

Let $p$ be an infinite irrational path. We wish to show that $p$ contains a
line point. Assume, by way of contradiction, that no vertex on $p$ is a line
point. Then there will be infinitely many bifurcating vertices on $p$.

By Lemma \ref{Disjoint cycles}, there are at most finite number of vertices on
$p$ which are bases for a finite number of closed paths. Then we can write
$p=\gamma q$ where $q$ is an infinite irrational path containing infinitely
many bifurcating vertices, and no vertex on the path $q$ lies on a closed path
and $\gamma$ is a finite path containing at most finitely many closed paths.

Let $q=e_{1}e_{2}e_{3}....$ with $s(e_{i})=$ $v_{i}$, for all $i\geq1$. We
wish to construct an infinite irrational path not tail-equivalent to $q$.
Start with a vertex $v_{1}$ on $q$. By Lemma \ref{cofinal}, there is a path
$\alpha_{1}$ with $s(\alpha_{1})=v_{1}$ and its last edge $f^{(1)}\neq e_{i} $
for all $i\geq1$. Now $\alpha_{1}$ is not a closed path, since $v_{1}$ does
not lie on a closed path. Apply Lemma \ref{cofinal} to $r(f^{(1)})=v_{2} $ to
get a path $\alpha_{2}$ with $s(\alpha_{2})=v_{2}$ and its last edge
$f^{(2)}\neq e_{i}$ for all $i\geq1$. If $\alpha_{2}$ is not a closed path
apply Lemma \ref{cofinal} to $r(\alpha_{2})$ to get a path $\alpha_{3}$ with
$s(\alpha_{3})=r(\alpha_{2})$ and its last edge is not an edge in\ $q$. If
$\alpha_{2}$ is a closed path, then by assumption, it has an exit, say $g_{2}%
$. Now $r(g_{2})$ does not lie on the path $\alpha_{2}$, since, otherwise,
Lemma \ref{Disjoint cycles}(i) would give rise to uncountably many irrational
paths not tail-equivalent to each other, a contradiction to our assumption.
Now apply Lemma \ref{cofinal} to $r(g_{2})$ to obtain a path $\alpha_{3}$ with
$s(\alpha_{3})=r(g_{2})$ and its last edge is not an edge in $q$. Proceeding
like this, we obtain an infinite path $q^{\prime}=\alpha_{1}\alpha_{2}%
\cdot\cdot\cdot\alpha_{n}\cdot\cdot\cdot$ which is not rational by
construction. Clearly, $q^{\prime}$ is not tail-equivalent to $q$. \ This
contradiction shows that $E$ contains a line \ point $u$ in which case $E^{0}$
is the hereditary saturated closure of $\{u\}$. This proves (b).

Now (b)$\implies$(c) follows from Theorem \ref{Generalized Hazrat}.

Assume (c). If $L\cong_{gr}M_{\Lambda}(\mathbb{K})(\bar{p})$ or if
$L\cong_{gr}M_{\Lambda}(\mathbb{K}[x^{m},x^{-m}])(\bar{p})$ with indicated
gradings, then, by Theorem \ref{graded socle}, $L$ is graded-semisimple being
a graded direct sum of isomorphic copies of $\mathbb{K}$ with trivial grading
or a graded direct sum of isomorphic copies of $\mathbb{K}[x,x^{-1}]$ with its
natural grading. In this case, it is clear that all the graded-simple
left/right $L$-modules are graded isomorphic. This proves (a).

\end{proof}

\begin{remark}
The conditions above are also equivalent to the following two conditions: (i) $L$ is a graded-simple ring which is graded-self-injective (by Theorem 6.17, \cite{HR}), (ii) $L$ is a graded-simple ring which is graded-semisimple (by Theorem \ref{graded socle}).
\end{remark}

It is possible for a Leavitt path algebra $L_{\mathbb{K}}(E)$ to have exactly
one graded-isomorphism class of graded-simple left modules, but have
infinitely non-isomorphic simple left modules. An easy example is given below.

\begin{example} \rm
Let $E$ be the graph consisting of a single loop $c$ based at a vertex
$v$ as shown below:
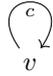
\[
\begin{tikzcd} v \arrow["c"', from=1-1, to=1-1, loop, in=55, out=125, distance=10mm] \end{tikzcd}
\]

\noindent Then $L_{\mathbb{K}}(E)\cong_{gr}\mathbb{K}[x,x^{-1}]$ under
the map sending $v\mapsto1,c\mapsto x$ and $x^{\ast}\mapsto x^{-1}$. Now
$\mathbb{K}[x,x^{-1}]$ is itself a graded-simple $\mathbb{K}[x,x^{-1}]$-module
and thus has exactly one graded-isomorphism class of graded-simple modules.
Since $\mathbb{K}[x,x^{-1}]$ is a principal ideal domain, for each of the
infinitely many irreducible polynomial $p(x)$ in $\mathbb{K}[x,x^{-1}] $,
$<p(x)>$ is a (non-graded) maximal ideal and we get a simple module
$\mathbb{K}[x,x^{-1}]/<p(x)>$. Thus $L_{\mathbb{K}}(E)$ has, up to a graded
isomorphism, only one graded-simple left module, but has infinitely many
non-graded simple module of the form $L_{\mathbb{K}}(E)/<p(c)>$ corresponding
to the infinitely many irreducible polynomials $p(x)\in\mathbb{K}[x,x^{-1}]$.

\end{example}

\begin{example} \rm
Consider the graph $E$ below. 

\bigskip

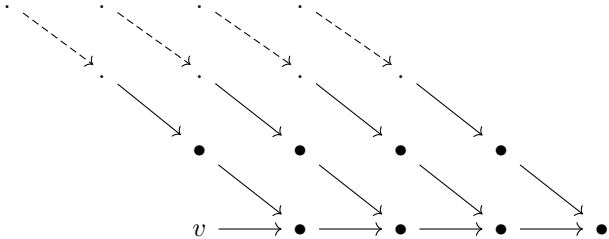
\begin{tikzcd}
	{.} & {.} & {.} & {.} \\
	& {.} & {.} & {.} & {.} \\
	&& \bullet & \bullet & \bullet & \bullet \\
	&& v & \bullet & \bullet & \bullet & \bullet
	\arrow[dashed, from=1-1, to=2-2]
	\arrow[dashed, from=1-2, to=2-3]
	\arrow[dashed, from=1-3, to=2-4]
	\arrow[dashed, from=1-4, to=2-5]
	\arrow[from=2-2, to=3-3]
	\arrow[from=2-3, to=3-4]
	\arrow[from=2-4, to=3-5]
	\arrow[from=2-5, to=3-6]
	\arrow[from=3-3, to=4-4]
	\arrow[from=3-4, to=4-5]
	\arrow[from=3-5, to=4-6]
	\arrow[from=3-6, to=4-7]
	\arrow[from=4-3, to=4-4]
	\arrow[from=4-4, to=4-5]
	\arrow[from=4-5, to=4-6]
	\arrow[from=4-6, to=4-7]
\end{tikzcd}

\noindent Clearly this graph $E$ is row finite, downward directed and $E^0$ is the hereditary saturated closure of vertex $v$ which is a line point. Therefore, by the above theorem $L_{\mathbb K}(E)$ has a unique isomorphism class of graded simple modules.

\end{example}

\begin{example} \rm 

Consider the graph $E$ below.

\smallskip

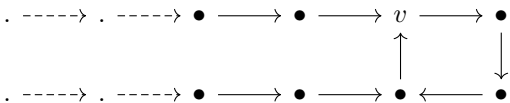
\begin{tikzcd}
	{.} & {.} & \bullet & \bullet & v & \bullet \\
	{.} & {.} & \bullet & \bullet & \bullet & \bullet
	\arrow[dashed, from=1-1, to=1-2]
	\arrow[dashed, from=1-2, to=1-3]
	\arrow[from=1-3, to=1-4]
	\arrow[from=1-4, to=1-5]
	\arrow[from=1-5, to=1-6]
	\arrow[from=1-6, to=2-6]
	\arrow[dashed, from=2-1, to=2-2]
	\arrow[dashed, from=2-2, to=2-3]
	\arrow[from=2-3, to=2-4]
	\arrow[from=2-4, to=2-5]
	\arrow[from=2-5, to=1-5]
	\arrow[from=2-6, to=2-5]
\end{tikzcd}

\noindent Clearly this graph $E$ is row finite, downward directed and $E^0$ is the hereditary saturated closure of vertex $v$ which lies on a cycle without exits. Therefore, by the above theorem $L_{\mathbb K}(E)$ has a unique isomorphism class of graded simple modules.

\end{example}

\section{Leavitt path algebras having countably many non-isomorphic
graded-simple modules}

\noindent In this section, we give a complete description of the Leavitt path
algebra $L=L_{\mathbb{K}}(E)$ of an arbitrary graph $E$ over a field
$\mathbb{K}$ having at most countably many graded-simple left modules which
are pair-wise not graded-isomorphic. Such a Leavitt path algebra $L$ is shown
to be the union of a smooth ascending chain of graded ideals%
\[
0=L_{0}\subseteq L_{1}\subseteq\cdot\cdot\cdot\subseteq L_{\alpha}\subseteq
L_{\alpha+1}\subseteq\cdot\cdot\cdot\qquad\qquad(\alpha<\tau)
\]
where $\tau$ is a countable ordinal, where for each $0\leq\alpha<\tau$,
$L_{\alpha+1}/L_{\alpha}\cong_{gr}(\oplus_{i\in X}M_{\Lambda_{i}}%
(\mathbb{K})(\bar{\delta}_{i}))\oplus(\oplus_{j\in Y}M_{\Upsilon_{i}%
}(\mathbb{K})(\bar{\sigma}_{j}))$ where $X,Y,\Lambda_{i},\Upsilon_{j}$ are
suitable index sets.

\begin{lemma}
\label{NE cycle or line points} Suppose $L=L_{\mathbb{K}}(E)$ has at most
countably many non-isomorphic graded simple left modules. Then the graph
contains line points or cycles without exits.
\end{lemma}

\begin{proof}
Assume, by way of contradiction, that the graph $E$ contains no line points
and no cycles without exits. In particular, $E$ has no sinks. As every cycle
has an exit, every vertex sits on an infinite path which is irrational. Since
there are no line points, every infinite irrational path will have infinitely
many bifurcating vertices. Let $\alpha_{1},\alpha_{2},\cdot\cdot\cdot
\alpha_{n},\cdot\cdot\cdot$ be a listing of representatives of the countably
many equivalent classes of tail-equivalent infinite irrational paths in $E$.
In view of Lemma \ref{Disjoint cycles} (ii), we may assume that no vertex on
each of the path $\alpha_{i}$ is the base of a closed path. We wish to
construct an infinite irrational path not tail-equivalent to any of these
paths $\alpha_{i}$ by using Lemma \ref{cofinal} repeatedly. Start with a
vertex $v$. By Lemma \ref{cofinal}, there is a path $p_{11}$ with
$s(p_{11})=v$ such that the last edge of $p_{11}$ is not an edge in the path
$\alpha_{1}$. By Lemma \ref{cofinal}, there is a path $p_{21}$ with
$s(p_{21})=r(p_{11})$ such that the final edge of $p_{21}$ is not an edge of
the path $\alpha_{1}$. Applying Lemma \ref{cofinal} \ to $r(p_{21})$, we have
a path $p_{22}$ with $s(p_{22})=r(p_{21})$ such that the last edge of $p_{22}$
is not an edge in the path $\alpha_{2}$. Again, there is a path $p_{31}$ with
$s(p_{31})=r(p_{22})$ such that the last edge of $p_{31}$ is not an edge in
the path $\alpha_{1}$. Let $p_{32}$ be a path with $s(p_{32})=r(p_{31})$ such
that the last edge of $p_{32}$ is not an edge in the path $\alpha_{2}$. Let
$p_{33}$ be a path $s(p_{33})=r(p_{32})$ such that the last edge of $p_{33}$
is not an edge of the path $\alpha_{3}$. Again, let $p_{41}$ be a path with
$s(p_{41})=r(p_{33})$ such that the last edge of $p_{41}$ is not an edge in
the path $\alpha_{1}$. Let $p_{42}$ be a path with $s(p_{42})=r(p_{41})$ such
that the last edge of $p_{42}$ is not an edge in the path $\alpha_{2}$. The
path $p_{43}$ satisfies $s(p_{43})=r(p_{42})$ with last edge of $p_{43}$ is
not an edge in the path $\alpha_{3}$. Likewise, we can get a path $p_{44}$
with $s(p_{44})=r(p_{43})$ such that the last edge of $p_{44}$ is not an edge
in the path $\alpha_{4}$. Proceeding like this, we get an infinite irrational
path%
\[
\beta=p_{11}p_{21}p_{22}p_{31}p_{32}p_{33}p_{41}p_{42}p_{43}p_{44}p_{51}%
\cdot\cdot\cdot
\]
where, for all $i,j\in%
\mathbb{N}
$, the final edge of $p_{ij}$is not an edge in the path $\alpha_{j}$. Clearly,
$\beta$ is not tail-equivalent to any of the paths $\alpha_{i},i\in%
\mathbb{N}
$. This contradiction shows that the $E$ contains line points or cycles
without exits.
\end{proof}

\begin{theorem}
\label{Graded CIRT} Let $E$ be an arbitrary graph and $\mathbb{K}$ be a field.
Then the following are equivalent for $L=L_{\mathbb{K}}(E):$

(a) There are at most countably many isomorphism classes of graded-simple left
$L$-modules.

(b) $L$ is the union of a smooth well-ordered ascending chain of graded
ideals
\[
0\subseteq I_{1}\subseteq\cdot\cdot\cdot\subseteq I_{\alpha}\subseteq
I_{\alpha+1}\subseteq\cdot\cdot\cdot\qquad\qquad(\alpha<\tau)\qquad\qquad
(\ast)
\]
where $\tau$ is a countable ordinal and for each $0\leq\alpha<\tau$,
$I_{\alpha+1}/I_{\alpha}\cong_{gr}(\oplus_{i\in X}M_{\Lambda_{i}}%
(\mathbb{K})(\bar{\delta}_{i}))\oplus(\oplus_{j\in Y}M_{\Upsilon_{j}%
}(\mathbb{K}[x^{m_{j}},x^{-m_{j}}])(\bar{\sigma}_{j}))$ with appropriate
matrix gradings, where $\Lambda_{i},\Upsilon_{j}$ are suitable index sets and
$X,Y $ are at most countable.
\end{theorem}

\begin{proof}
Assume (a). By Lemma \ref{NE cycle or line points}, the graph $E$ contains
line points and/or cycles without exits. Let $I_{1}$ be the ideal generated by
all the line points and vertices on all the cycles without exits in $E$. Then
$I_{1}=I(H,\emptyset)$, where $H=I_{1}\cap E^{0}$. By Theorem
\ref{graded socle}, $I_{1}$ is the graded socle of $L$ and is a graded direct
sum of at most countable number of matrices of the form $M_{\Lambda_{1}%
}(\mathbb{K})(\bar{\delta}_{1})$ and $M_{\Upsilon_{1}}(\mathbb{K}[x^{m_{1}%
},x^{-m_{1}}])(\bar{\sigma}_{1})$. Suppose we have defined a graded ideal
$I_{\alpha}=I(H_{\alpha},S_{\alpha})$ for some $\alpha\geq1$, where
$H_{\alpha}=I_{\alpha}\cap E^{0}$ and $S_{\alpha}=\{v\in B_{H_{\alpha}%
}:v^{H_{\alpha}}\in I_{\alpha}\}$. If $I_{\alpha}\neq L$, consider
$L/I_{\alpha}\cong L_{\mathbb{K}}(E\backslash(H_{\alpha},S_{\alpha})$. \ Now
$L/I_{\alpha}$ has at most countable number of non-isomorphic graded-simple
left $L$-modules. Hence, by Lemma \ref{NE cycle or line points},
$E\backslash(H_{\alpha},S_{\alpha})$ contains line points or vertices on
cycles without exits and the ideal generated by these vertices is
$\Soc_{gr}(L/I_{\alpha})$. Let $I_{\alpha+1}$ be the ideal of $L$ such that
$I_{\alpha+1}/I_{\alpha}=\Soc_{gr}(L/I_{\alpha})$. By Theorem
\ref{graded socle}, $I_{\alpha+1}/I_{\alpha}$ is a graded direct sum of at
most countable number of matrices of the form $M_{\Lambda_{\alpha}}%
(\mathbb{K})(\bar{\delta}_{\alpha})$ and $M_{\Upsilon_{\alpha}}(\mathbb{K}%
[x^{m_{1}},x^{-m_{1}}])(\bar{\sigma}_{\alpha})$. If $\gamma$ is a limit
ordinal and $I_{\alpha}$ has been defined for all $\alpha<\gamma$, then define
$I_{\gamma}=\cup_{\alpha\,\gamma}I_{\alpha}$. By transfinite induction, we
then obtain a smooth ascending chain (*) of graded ideals with the desired
properties and its union is $L$. Now each graded ideal is isomorphic to the
Leavitt path algebra of a suitable graph (see [15]) and so is a ring with
local units. Hence every graded left/right ideal of $I_{\alpha+1}$ is also a
graded left/right ideal of $L$. It is then readily seen that every
graded-simple left $I_{\alpha+1}$-module is isomorphic to a graded-simple left
$L$-module. Since $L$ has at most countable number of non-isomorphic
graded-simple left modules, the length of the chain $\tau$ is countable and
that, for each $\alpha<\tau$, $I_{\alpha+1}/I_{\alpha}$ is a graded direct sum
of at most countable number of matrices of the form $M_{\Lambda_{\alpha}%
}(\mathbb{K})(\bar{\delta}_{\alpha})$ and $M_{\Upsilon_{\alpha}}%
(\mathbb{K}[x^{m_{1}},x^{-m_{1}}])(\bar{\sigma}_{\alpha})$. This proves (b).

Assume (b), so that $L$ is the union of a chain of graded ideals satisfying
the stated properties. Let $S=L/M$ be a graded-simple module where $M$ is a
maximal graded left ideal of $L$. Let $\beta$ be the smallest ordinal such
that $I_{\beta}\nsubseteq M$. Clearly $\beta$ is a non-limit ordinal. Let
$\beta=\alpha+1$. Then $I_{\alpha}\subseteq M$. So,%
\[
S=(M+I_{\alpha+1})/M\cong_{gr}I_{\alpha+1}/(I_{\alpha+1}\cap M)\cong%
_{gr}(I_{\alpha+1}/I_{\alpha})/[(I_{\alpha+1}\cap M)/I_{\alpha}].
\]
Thus every graded-simple left $L$-module is isomorphic to a graded-simple left
module over the ring $I_{\alpha+1}/I_{\alpha}$ for some $\alpha$ with
$0\leq\alpha<\tau$. Now for each $\alpha<\tau$, $I_{\alpha+1}/I_{\alpha}$ has
at most countably many non-isomorphic graded simple left modules, due to the
fact that $I_{\alpha+1}/I_{\alpha}$ is a graded direct sum of at most
countably many matrix rings $M_{\Lambda_{\alpha}}(\mathbb{K})(\bar{\delta
}_{\alpha})$ and $M_{\Upsilon_{\alpha}}(\mathbb{K}[x^{m_{1}},x^{-m_{1}}%
])(\bar{\sigma}_{\alpha})$, where $\Lambda_{i},\Upsilon_{j}$ are suitable
index set. Since $\tau$ is a countable ordinal, it is readily seen that $L$
has at most countable number of non-isomorphic graded-simple left $L$-modules.
This proves (a).
\end{proof}

\section{Examples}

\noindent In this section, we illustrate the main results of this paper, by constructing
several graphs $E$ for which the corresponding Leavitt path algebra $L_{\mathbb K}(E)$
satisfies the desired property.

\begin{example} \rm
Let $G_{1}$ be the graph
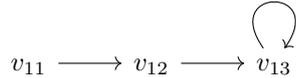
\[\begin{tikzcd}
	{v_{11}} & {v_{12}} & {v_{13}}
	\arrow[from=1-1, to=1-2]
	\arrow[from=1-2, to=1-3]
	\arrow[from=1-3, to=1-3, loop, in=55, out=125, distance=10mm]
\end{tikzcd}\]
$G_{1}$ contains no line points , but contains a single loop
$c$ (without exits) at the vertex $v_{13}$. Now $v_{11},v_{12},v_{13}$ are all
Laurent vertices and $(G_{1})^{0}$ is the hereditary saturated closure of
$\{v_{13}\}$. There are three paths ending at $v_{13}$\ that do not go through
the loop \ $c$ and have length $0,1,2$ respectively. Hence, $L_{\mathbb K}(G_{1}%
)\cong_{gr}M_{3}(\mathbb K[x,x^{-1}](0,1,2)$. By Theorem \ref{Only one graded-simple},
all graded-simple modules over $L_{\mathbb K}(G_{1})$ are graded isomorphic to the
graded-simple module $\mathbb K[x,x^{-1}]$ and are faithful $L_{\mathbb K}(G_{1})$-modules.
Thus $L_{\mathbb K}(G_{1})$ is a graded direct sum of faithful graded-simple left
$L_{\mathbb K}(G_{1})$-modules.

\end{example}

\begin{example} \rm 
 Let $G_{2}$ be the graph
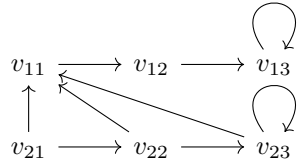
\[\begin{tikzcd}
	{v_{11}} & {v_{12}} & {v_{13}} \\
	{v_{21}} & {v_{22}} & {v_{23}}
	\arrow[from=1-1, to=1-2]
	\arrow[from=1-2, to=1-3]
	\arrow[from=1-3, to=1-3, loop, in=55, out=125, distance=10mm]
	\arrow[from=2-1, to=1-1]
	\arrow[from=2-1, to=2-2]
	\arrow[from=2-2, to=1-1]
	\arrow[from=2-2, to=2-3]
	\arrow[from=2-3, to=1-1]
	\arrow[from=2-3, to=2-3, loop, in=55, out=125, distance=10mm]
\end{tikzcd}\]

Now $G_{2}$
contains no line points and has a single loop without
exits and the Laurent vertices in $G_{2}$ are the vertices in the first row,
namely, $v_{11},v_{12},v_{13}$. $H_{1}=\{v_{11},v_{12},v_{13}\}$ is a
hereditary saturated subset of $(G_{2})^{0}$. The graded ideal $<H_{1}>=S_{1}$,
the graded socle of $L_{\mathbb K}(G_{2})$ and it is clear from the explanation of
$L_{\mathbb K}(G_{1})$, that $S_{1}$ is a graded direct sum of isomorphic faithful
graded-simple $L_{\mathbb K}(G_{2})$-modules . Now the quotient graph $G_{2}\backslash
H_{1}=G_{1}$ and so $L_{\mathbb K}(G_{2})/S_{1}\cong_{gr}L_{\mathbb K}(G_{1})$. It is then
clear that $L_{\mathbb K}(G_{2})/S_{1}$ is a graded direct sum of graded-isomorphic
graded simple modules annihilated by the ideal $S_{1}$. Thus $L_{\mathbb K}(G_{2})$
has exactly two distinct graded-isomorphism classes of graded-simple left
modules. 

\end{example}

\begin{example} \rm

Let $G_{3}$ be the graph
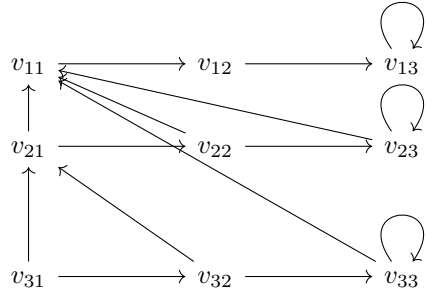
\[\begin{tikzcd}
	{v_{11}} && {v_{12}} && {v_{13}} \\
	{v_{21}} && {v_{22}} && {v_{23}} \\
	\\
	{v_{31}} && {v_{32}} && {v_{33}}
	\arrow[from=1-1, to=1-3]
	\arrow[from=1-3, to=1-5]
	\arrow[from=1-5, to=1-5, loop, in=55, out=125, distance=10mm]
	\arrow[from=2-1, to=1-1]
	\arrow[from=2-1, to=2-3]
	\arrow[from=2-3, to=1-1]
	\arrow[from=2-3, to=2-5]
	\arrow[from=2-5, to=1-1]
	\arrow[from=2-5, to=2-5, loop, in=55, out=125, distance=10mm]
	\arrow[from=4-1, to=2-1]
	\arrow[from=4-1, to=4-3]
	\arrow[from=4-3, to=2-1]
	\arrow[from=4-3, to=4-5]
	\arrow[from=4-5, to=1-1]
	\arrow[from=4-5, to=4-5, loop, in=55, out=125, distance=10mm]
\end{tikzcd}\]

As before the Laurent vertices in $G_{3}$ are the vertices  $v_{11},v_{12},v_{13}$  on the
first row and they generate the graded socle $S_{1}$ which is a graded direct
sum of  graded-isomorphic faithful graded-simple $L_{\mathbb K}(G_{3})$-modules and
$L_{\mathbb K}(G_{3})/S_{1}\cong_{gr}L_{\mathbb K}(G_{2})$. Let $S_{2}$ be the graded ideal
containing $S_{1}$ such that $S_{2}/S_{1}\cong_{gr}Soc_{gr}(L_{\mathbb K}(G_{2}))$,
the graded-socle of $L_{\mathbb K}(G_{2})$. Using the fact every graded $(L_{\mathbb K}%
(G_{3})/S_{1})$-module is a graded $L_{\mathbb K}(G_{3})$-module annihilated by
$S_{1}$, we conclude that $S_{2}/S_{1}$ a graded-direct sum of
graded-isomorphic graded-simple $L_{\mathbb K}(G_{3})$-module annihilated by $S_{1}$.
Also
\[
L_{\mathbb K}(G_{3})/S_{2}\cong_{gr}(L_{\mathbb K}(G_{3})/S_{1})/(S_{2}/S_{1})\cong_{gr}%
(L_{\mathbb K}(G_{2})/Soc_{gr}(L_{\mathbb K}(G_{2}))\cong_{gr}L_{\mathbb K}(G_{1})
\]
is a graded direct sum graded-isomorphic graded-simple left  $L_{\mathbb K}(G_{3}%
)$-modules annihilated by the ideal $S_{2}$. Thus $L_{\mathbb K}(G_{3})$ has exactly
three distinct graded-isomorphism classes of graded-simple left modules.
\end{example}

\begin{example} \rm
Proceeding as above in the previous examples, we can construct, by simple induction, a graph $G_{n}$
for every $n\geq1$ such that $L_{\mathbb K}(G_{n})$ has exactly $n$ graded isomorphism
classes of graded-simple left $L_{\mathbb K}(E)$-modules. Observing that, for all
$i\geq1$,  we have a natural embedding of $G_{i}$ into $G_{i+1}$, we obtain a
chain of graphs
\[
G_{1}\subseteq G_{2}\subseteq\cdot\cdot\cdot\subseteq G_{n}\subseteq
\cdot\cdot\cdot.
\]
If we set  $G_{\omega}=\cup_{n\geq1}G_{n}$, then $L_{\mathbb K}(G_{\omega})$ will have
exactly countable number of distinct graded isomorphism classes of
graded-simple left $L_{\mathbb K}(G_{\omega})$-modules.
\end{example}

\begin{example} \rm
Now the graph $G_{\omega}$ does not contain any line points. One can modify
the graph $G_{\omega}$ to obtain an infinite graph $H_{\omega}$ containing
line points by replacing each odd numbered row in $G_{\omega}$ by $%
\begin{array}
[c]{ccccc}%
\bullet_{{}} & \longrightarrow & \bullet_{{}} & \longrightarrow &
\bullet\longrightarrow\cdot\cdot\cdot
\end{array}
$ as given below:
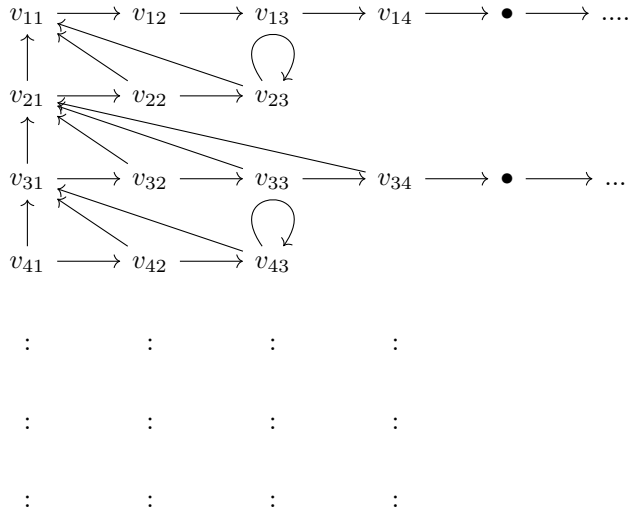
\[\begin{tikzcd}
	{v_{11}} & {v_{12}} & {v_{13}} & {v_{14}} & \bullet & {....} \\
	{v_{21}} & {v_{22}} & {v_{23}} \\
	{v_{31}} & {v_{32}} & {v_{33}} & {v_{34}} & \bullet & {...} \\
	{v_{41}} & {v_{42}} & {v_{43}} \\
	{:} & {:} & {:} & {:} \\
	{:} & {:} & {:} & {:} \\
	{:} & {:} & {:} & {:}
	\arrow[from=1-1, to=1-2]
	\arrow[from=1-2, to=1-3]
	\arrow[from=1-3, to=1-4]
	\arrow[from=1-4, to=1-5]
	\arrow[from=1-5, to=1-6]
	\arrow[from=2-1, to=1-1]
	\arrow[from=2-1, to=2-2]
	\arrow[from=2-2, to=1-1]
	\arrow[from=2-2, to=2-3]
	\arrow[from=2-3, to=1-1]
	\arrow[from=2-3, to=2-3, loop, in=55, out=125, distance=10mm]
	\arrow[from=3-1, to=2-1]
	\arrow[from=3-1, to=3-2]
	\arrow[from=3-2, to=2-1]
	\arrow[from=3-2, to=3-3]
	\arrow[from=3-3, to=2-1]
	\arrow[from=3-3, to=3-4]
	\arrow[from=3-4, to=2-1]
	\arrow[from=3-4, to=3-5]
	\arrow[from=3-5, to=3-6]
	\arrow[from=4-1, to=3-1]
	\arrow[from=4-1, to=4-2]
	\arrow[from=4-2, to=3-1]
	\arrow[from=4-2, to=4-3]
	\arrow[from=4-3, to=3-1]
	\arrow[from=4-3, to=4-3, loop, in=55, out=125, distance=10mm]
\end{tikzcd}\]

Again, $L_{\mathbb K}(H_{\omega})$ has exactly
countably many graded-simple left $L_{\mathbb K}(H_{\omega})$-modules no two of
which are graded-isomorphic.
\end{example}

\begin{example} \rm
By transfinite induction, one can then construct, for each infinite ordinal
$\alpha$, the graphs $G_{\alpha},H_{\alpha}$ such that the corresponding
Leavitt path algebras  $L_{\mathbb K}(G_{\alpha}),L_{\mathbb K}(H_{\alpha})$ each possesses
exactly $\alpha$ number of graded-isomorphism classes of graded-simple left modules.
\end{example}

\section{Naimark's Problem For Associative Algebras}

\noindent In general, the algebraic version of Naimark's problem has a
negative answer for arbitrary associative algebras over a field, as justified
by the following example, by J.H. Cozzens, of a countable dimensional algebra
$A$ over a field which is a domain but all the simple left $A$-modules are isomorphic.

\begin{example}
\label{Cozzens}(J.H. Cozzens \cite{Cozzens} ). Let $\mathbb{K}$ be a field
with derivation $D$ which is a ``universal differential field" in the sense
that the equation
\[
p(x,D(x),\cdot\cdot\cdot,D^{n}(x))=0
\]
has a solution $\xi\in\mathbb{K}$ \ for all non-constant polynomials
$p(x)\in\mathbb{K}[x]$. Let $A=\mathbb{K}[x,D]$ be the algebra of differential
polynomials over $\mathbb{K}$. Theorem 1.4, \cite{Cozzens} states that $A$ has
exactly one isomorphism class of simple left modules and that $A$ is a
principal left/right ideal domain which is not a division ring (and, in
particular, not von Neumann regular). The fact that $A$ is a domain which is
not a division ring implies that $A$ cannot be isomorphic to a subalgebra of
linear transformations of a vector space over a field.
\end{example}

The next Proposition gives a necessary condition and a partial converse for an
associative algebra $A$ to have a unique simple module up to isomorphism.

\begin{proposition}
\label{AA}Suppose $A$ is a unital associative algebra over a field
$\mathbb{K}$.

(a) If $A$ has, up to isomorphism, exactly one simple left (right) $A$-module
$S$, then the Jacobson radical $J(A)$ is the unique maximal ideal of $A$
containing all the other ideals of $A$. In this case, either (i) $A/J(A)\cong
M_{n}(D)$, the matrix ring of $n\times n$ matrices over the division ring
$D\cong End_{A}(S)$ for some $n\geq1$, or\ (ii) $A/J(A)$ is a unital simple
ring with $\Soc(A/J(A))=0$.

(b) Converse of (a)(i): If $A/J(A)\cong M_{n}(D)$, the matrix ring of $n\times
n$ matrices over a division ring $D$, then $A$ will have exactly one simple
left/right module up to isomorphism.
\end{proposition}

\begin{proof}
(a) Now $A$ has exactly one primitive ideal $M$ being the annihilator of the
unique simple left (right) $A$-module $S$ and hence $M=J(A)$, the Jacobson
radical of $A$. We claim that every two-sided ideal $I$ of $A$ must be
contained in $J(A)$. Indeed, if $I\nsubseteq J(A)$, then a simple left/right
module $T$ over the unital ring $A/I$ will be a simple $A$-module not
isomorphic to $S$, since the primitive ideal $ann(_{A}T)\supseteq I$ and so
$ann(_{A}T)\neq J(A)=ann(_{A}S)$, a contradiction. In particular, $J(A)$ must
be a unique maximal ideal of $A$. If $\Soc(A/J(A))\neq0$, then, by simplicity,
$A/J(A)=\Soc(A/J(A))$ and will be a direct sum of finitely many (as $A/J(A)$
is unital) isomorphic simple left modules. Consequently, $A/J(A)\cong
M_{n}(D)$, the ring of $n\times n$ matrices over the division ring $D\cong
End_{A}(S)$. Otherwise, $A/J(A)$ is a unital simple ring with $Soc(A/J(A))=0$.

(b) The proof follows from the fact that all the simple left (right) modules
over $M_{n}(D)\cong A/J(A)$ are isomorphic, that $J(A)$ is the unique
primitive ideal of $A$ and that the simple left (right) $A$-modules are
precisely the simple left (right) $A/J(A)$-modules.
\end{proof}

\begin{remark}
In general, the converse of Proposition \ref{AA} does not hold. If
$\mathbb{K}$ is a field of characteristic $0$, then the Weyl algebra
$A_{1}(\mathbb{K}) $ is a simple ring with zero socle, but there are
infinitely many non-isomorphic simple left modules over $A_{1}(\mathbb{K})$.
Specifically, recalling that $R=A_{1}(\mathbb{K})$ is the $\mathbb{K}$-algebra
generated by $x,y$ subject to the relation $xy-yx=1$, one shows that, for each
non-zero polynomial $p(y)\in\mathbb{K}[y]$, the left ideal $R(x-p(y))$ is
maximal giving rise to a simple left $R$-module $S=R/(R(x-p(y)))$. (See
Corollary 3.17, Exercise 3.18, \cite{L}).
\end{remark}

\begin{corollary}
\label{Countable Algebra} Suppose $A$ is a unital countable dimensional
$\mathbb{K}$-algebra in which every non-zero one-sided ideal contains a
non-zero idempotent. Then the following are equivalent:

(a) Any two simple left $A$-modules are isomorphic;

(b) $A\cong M_{n}(D)$ where $D$ is a division ring and $n\geq1$.
\end{corollary}

\begin{proof}
Assume (a). Since every non-zero one-sided ideal of $A$ contains a non-zero
idempotent, the Jacobson radical $J(A)=0$. Hence, by Proposition \ref{AA}, $A
$ is a simple ring such that either $\Soc(A)=0$ or $A\cong M_{n}(D)$ where $D
$ is a division ring and $n\geq1$. If $\Soc(A)=0$, then, Rosenberg has shown
in (\cite{Rosenberg}, Theorem 3) that $A$ will have infinitely many
non-isomorphic simple left modules. Hence $A\cong M_{n}(D)$. This proves (b).

(b)$\implies$(a) is well-known (see \cite{L}).
\end{proof}

\textbf{Remark}: We do not know if the above corollary holds if the algebra
$A$ has uncountable $\mathbb{K}$-dimension.

\begin{proposition}
Suppose $\mathbb{K}$ is a field and $\Gamma$ is a group. Suppose $A$ is a
unital $\Gamma$-graded associative $\mathbb{K}$-algebra having exactly one
graded-isomorphism class of graded-simple left $A$-modules. Then either
$A/J_{gr}(A)\cong_{gr}M_{n}(K)(\bar{\sigma})$ where $\bar{\sigma}=(\sigma
_{1},\cdot\cdot\cdot,\sigma_{n})\in\Gamma^{n}$ or $A/J_{gr}(A)$ is a
graded-simple ring with $\Soc_{gr}(A/J_{gr}(A))=0$.
\end{proposition}

\begin{proof}
The proof follows the same \ argument used in the proof of Proposition
\ref{AA}. Since all the graded-simple left $A$-modules are isomorphic, there
can be only one annihilating ideal of the graded simple left modules in $A$.
Now the graded Jacobson radical $J_{gr}(A)$ is the intersection of all the
annihilators of the graded-simple left $A$-modules (Lemma 1.7.4 \cite{NO}).
Hence $J_{gr}(A)$ is the only annihilator of graded-simple left $A$-modules.
We claim that $J_{gr}(A)$ contains every proper graded ideal of $A$. Indeed,
if $I$ is a non-zero graded ideal with $I\nsubseteq J_{gr}(A)$, then a
graded-simple left module\ over the quotient algebra $A/I$ will also be a
graded-simple left $A$-module whose annihilating ideal $Q$ contains $I$ and
$Q\neq J_{gr}(A)$. This contradiction shows that $J_{gr}(A)$ is a maximal
graded ideal containing every proper graded ideal of $A$. So $A/J_{gr}(A)$ is
a graded-simple ring. If the graded-socle $\Soc_{gr}(A/J_{gr}(A))\neq0$, then
$A/J_{gr}(A)\cong_{gr}M_{n}(K)(\bar{\sigma})$, being graded-simple and
graded-semisimple (Theorem I.5.8, \cite{NO}), where $\bar{\sigma}=(\sigma
_{1},\cdot\cdot\cdot,\sigma_{n})\in\Gamma^{n}$. Otherwise, $A/J_{gr}(A)$ is a
graded-simple ring with $\Soc_{gr}(A/J_{gr}(A))=0$.
\end{proof}

\end{document}